\documentclass[reqno]{amsart}
\usepackage{latexsym}
\usepackage{amssymb,amsthm,amsmath}
\usepackage[dvips]{graphicx}

\usepackage{color}
\usepackage[colorlinks=true]{hyperref}
\hypersetup{urlcolor=blue, citecolor=blue}
\usepackage{tikz}
\usepackage{multicol}

\theoremstyle{plain}
\newtheorem{theorem}{Theorem}

\newtheorem{corollary}[theorem]{Corollary}
\newtheorem{proposition}[theorem]{Proposition}

\theoremstyle{definition}
\newtheorem{definition}[theorem]{Definition}
\theoremstyle{remark}
\newtheorem{remark}[theorem]{Remark}

\def\d#1{{#1\kern-0.4em\char"16\kern-0.1em}}
\def\D#1{{\raise0.2ex\hbox{-}\kern-0.4em #1}}
\newcounter{zd}

\newcounter{zdr}[subsection]

\def\Div{{\rm div}}

\newcommand{\eps}{\varepsilon}

\def\pa{\partial}

\def\cal{\mathcal}

\def\R{{\bf R}}

\def\N{{\bf N}}

\def\eps{\varepsilon}

\def\mx{{\bf x}}

\def\my{{\bf y}}

\begin{document}
\title[Transport-collapse scheme]{Transport-collapse scheme for scalar conservation laws -- initial and boundary value problems}

\author{ Darko Mitrovi\'c}
\address{ Darko Mitrovi\'c  \newline
University of Zagreb, Faculty of Mathematics and Natural Sciences, Zagreb, Croatia} \thanks{Permanent address of Darko Mitrovic is Faculty of Mathematics and Natural Sciences, University of Montenegro, Podgorica, Montenegro}
\email{  darko.mitrovic@math.hr}

\author{ Andrej Novak}
\address{ Andrej Novak  \newline
University of Zagreb, Faculty of Electrical Engineering, Zagreb, Croatia}
\email{andrej.novak@fer.hr}

 \date{}

\begin{abstract}
We extend Brenier's transport collapse scheme on the Cauchy problem for heterogeneous
scalar conservation laws and initial-boundary value problem for homogeneous scalar conservation laws. It is based
on averaging out the solution to the corresponding kinetic equation,
and it necessarily converges toward the entropy admissible solution. In the case of initial-boundary value problems, such a procedure is used to construct a numerical scheme which leads to a new solution concept for the initial-boundary value problems for scalar conservation laws. We also provide numerical examples.
\end{abstract}



\maketitle

\section{Introduction}

The subject of the paper is the construction of new numerical method for Cauchy and initial-boundary problems for scalar conservation laws. The method is a generalization of the transport-collapse scheme introduced in \cite{brenier}. A consequence of the analysis of the scheme is a new solution concept of the initial-boundary value problem for scalar conservation laws, and this is the most important contribution of the paper.

In order to introduce it, let $\Omega\subset \R^d$ be a bounded smooth domain and $\R^+=[0,\infty)$.  We consider 
\begin{align}
\label{cl1-hom}
&\pa_t u+\Div_\mx f(u)=0,  \ \ (t,\mx)\in \R^+\times \Omega,\\
\label{ic1} &u|_{t=0}=u_0(\mx),\\
\label{bc1} &u|_{\R^+\times \pa \Omega}=u_B(t,\mx).
\end{align} where $f\in C^2(\R;\R^d)$. If not stated otherwise, we assume that $u_0 \in L^1(\R^d) \cap BV(\R^d)$, $u_B\in L^1_{loc}(\R^+\times \pa \Omega) \cap BV(\R^+\times \pa \Omega)$. We also assume that 
\begin{equation}
\label{bndass}
a \leq u_0, u_B \leq b \text{   for some constants $a \leq b$.}
\end{equation}


A typical problem described by \eqref{cl1-hom}, \eqref{ic1}, \eqref{bc1}
arises e.g. in traffic flow models. Namely, if we aim to describe a
flow on a finite highway (required to model on and off ramps) we
need to use boundary conditions \cite{SB}. For instance,
optimization of travel time and cost between two points can be
obtained by controlling incoming and outgoing car densities
\cite{Anc}.

Nevertheless, it is clear that the boundary conditions cannot be
prescribed if the characteristics corresponding to equation
\eqref{cl1}  and emerging from the boundary leave $\Omega$. 
This means that one needs to
introduce a new concept defining what conditions the
unknown function $u$ should satisfy in order to be a solution to \eqref{cl1-hom},
\eqref{ic1}, \eqref{bc1}. This was first done in \cite{BRN} where the existence of strong  traces at the boundary of solutions is assumed; see \cite{AM_jhde, pan, Vass}. The weak formulation which does not require existence of strong traces was later proposed by F.Otto \cite{Otto} and the corresponding numerical method was developed in \cite{Vov}. The concept is further extended in \cite{pan_tams} in a more general setting (on manifolds necessarily implying that the flux depends on $\mx$). Let us recall it here.

\begin{definition}
\label{def-panov}A function $u\in L^\infty(\Omega)$ is said to be the weak entropy solution to \eqref{cl1-hom}, \eqref{ic1}, \eqref{bc1} if there exists a constant $L\in \R$ such that for every $k\in \R$ and every non-negative $\varphi\in C_c(\R^d_+;\R^+)$, $\R^d_+=\R^+\times \R^d$, it holds

\begin{align}
\label{semi+}
&\int_{\R^d_+}\left(|u-k|_+ \pa_t \varphi+{\rm sgn}_+(u-k)(f(u)-f(k)) \, \nabla_{\mx} \varphi \right)d\mx dt\\&\qquad\qquad + \int_{\R^d}|u_0-k|_+\varphi(0,\cdot) d\mx + L \int_{\R^+\times \pa \Omega} \varphi \, |u_B-k|_+ d\gamma(\mx) dt \geq 0, \nonumber
\end{align}and

\begin{align}
\label{semi-}
&\int_{\R^d_+}\left(|u-k|_- \pa_t \varphi+{\rm sgn}_-(u-k)(f(u)-f(k)) \, \nabla_{\mx} \varphi \right)d\mx dt\\&\qquad\qquad + \int_{\R^d}|u_0-l|_-\varphi(0,\cdot) d\mx+  L \int_{\R^+\times  \pa \Omega} \varphi\, |u_B-k|_-  d\gamma(\mx) dt \geq 0,& \nonumber
\end{align} where $\gamma$ is the measure on $\pa \Omega$.
\end{definition}
As it comes to the refinement of the latter concept, we actually base it on an interesting observation from \cite{pan_tams} roughly stating that if the characteristics emerging from $\{t=0\}\times \Omega$ hit the boundary then the corresponding boundary value should not affect the solution. Thus, we are going to construct the definition of solution so that it involves somehow only those parts of the boundary which essentially influence on solutions. 

To be more precise, assume that we are dealing with the flux depending on $\mx$  i.e. $f=f(\mx,\lambda)$. Denote by 
$$
S^-=\{\mx \in \pa \Omega: \; \langle f'_\lambda(\mx,\lambda), \vec{\nu} \rangle \leq 0 \ \ \text{a.e.} \ \ \lambda \in I \},
$$ where $I$ contains all essential values of the functions $u_B$ and $u_0$ (i.e. of appropriate entropy solution $u$), and $\vec{\nu}$ is the outer unit normal on $\pa \Omega$. The set $S^-$ actually consists of all points such that all possible characteristics from that point enter into the (interior of the) set $\Omega$. Therefore, for every $\mx \in S^-$, the trace of the corresponding entropy solution is actually $u_B(\mx)$. 

Similarly, for 

$$
S^+=\{\mx \in \pa \Omega: \; \langle f'_\lambda(\mx,\lambda), \vec{\nu} \rangle \geq 0 \ \ \text{a.e.} \ \ \lambda \in I \},
$$ all possible characteristics issuing from $\mx\in S^+$ leave the set $\Omega$, and $u_B(\mx)$ does not influence on the weak entropy solution $u$ to \eqref{cl1-hom}, \eqref{ic1}, \eqref{bc1}.

However, both sets $S^-$ and $S^+$ can be empty since for some $\lambda \in I$, it can be $ \langle f'_\lambda(\mx,\lambda), \vec{\nu} \rangle \leq 0$ and for other $\lambda\in I$ we could have $\langle f'_\lambda(\mx,\lambda), \vec{\nu} \rangle > 0$. Therefore, in order to refine former arguments, we need to rewrite considered conservation laws so that we can more accurately take into account behaviour of the flux $f$ with respect to $\lambda$. A natural choice is the kinetic formulation to \eqref{cl1-hom} since it includes the variable $\lambda$ in a desired way. Before we introduce it, let us recall the Kruzhkov entropy admissibility conditions for (general, heterogeneous) scalar conservation laws. 

\begin{definition}
\label{def1} A bounded function $u$ is called an entropy admissible
solution to
\begin{equation}
\label{cl1}
\pa_t u+\Div_\mx f(t,\mx,u)=0
\end{equation} with the initial conditions \eqref{ic1} if for every convex function
$V\in C^2(\R)$, every $\lambda \in \R$ and every non-negative $\varphi \in
C^1_c(\R^d_+)$, it holds
\begin{align}
\label{a1} &\iint\limits_{\R^+\times \R^d}\! \big[ V(u)\pa_t
\varphi\!+\!\int_a^u f'_\lambda(t,\mx,v)\,V'(v)dv \cdot \nabla
\varphi\!+\!\int_a^u \Div_{\mx} f(t,\mx,v) V''(v)dv \,\varphi\big]
d\mx dt
\\& +\int_{\R^d} V(u_0(\mx))\varphi(0,\mx) d\mx\leq 0;
\nonumber
\end{align}
\end{definition} 

\begin{remark}
In the case of the heterogeneous equation \eqref{cl1}, we shall consider the Cauchy problem. As we shall see, the heterogeneity causes significant technical challenges. It is possible to overcome them in the case of the boundary problem as well, but we believe that this would be unnecessary complication which would hide main ideas of the new initial-boundary concept. In the numerical examples at the end of the paper, we shall take a space dependent flux in order to show how the method works in general situations.
\end{remark} 

Equivalent and more usual definition of admissible solution is
given by the Kruzhkov entropies $V(\lambda)=|u-\lambda|$, $\lambda \in
\R$, and it states that a bounded function $u$ is called an entropy
admissible solution to \eqref{cl1}, \eqref{ic1} if for
every $\lambda \in \R$ it holds
\begin{align}
\label{a1-stand} \pa_t |u-\lambda|+\Div_\mx [{\rm
sgn}(u-\lambda)(f(t,\mx,u)-f(t,\mx,\lambda))]+ {\rm
sgn}(u-\lambda)\Div_{\mx}f(t,\mx,\lambda) \leq 0
\end{align}in the sense of distributions on ${\cal D}'(\R^d_+)$, and
 it holds $esslim_{t\to 0}\int_\Omega |u(t,\mx)-u_0(\mx)|
d\mx=0$. Roughly speaking, by finding derivative with respect to $\lambda$ in \eqref{a1-stand} one reaches to the kinetic formulation provided below (see e.g. \cite{IV, LM2, LPT} for different variants).

\begin{theorem}
\label{db}\cite{dalibard}
The function $u\in C([0,\infty);L^1(\R^d))\cap
L^\infty_{loc}((0,\infty);L^\infty(\R^d))$ is the entropy admissible
solution to \eqref{cl1}, \eqref{ic1} if and only if there exists a
non-negative Radon measure $m(t,\mx,\lambda)$ such that
$m((0,T)\times \R^{d+1})<\infty$ for all $T>0$ and such that the
function $ \chi(\lambda,u)=\begin{cases}
1, & 0\leq \lambda \leq u\\
-1, & u \leq \lambda \leq 0\\
0, & else
\end{cases},
$ represents the distributional solution to
\begin{align}
\label{kin1} &\pa_t \chi +\Div_{(\mx,\lambda)}[ F(t,\mx,\lambda)
\chi]=\pa_\lambda m(t,\mx,\lambda), \ \ (t,\mx)\in \R^+\times \R^d,
\\
\label{ickin} &\chi(\lambda,u(t=0,\mx))=\chi(\lambda,u_0(\mx)),
\end{align} where $F=(f'_\lambda,-\sum\limits_{j=1}^d \pa_{x_j}f_j)$.
\end{theorem} In the next section, we shall provide properties of the function $\chi$.

Remark that through the kinetic concept, one reduces the nonlinear equation \eqref{cl1} on the linear (so called kinetic) equation (see Theorem \ref{db}). However,
derivative of a measure figures in the equation (see the right-hand
side of \eqref{kin1}) and it has one more variable which is usually called
kinetic or velocity variable. Due to the former reason, the kinetic
equation is not convenient for numerical implementation. Nevertheless, if we neglect the derivative of the measure, and then average
out the solution to the obtained linear equation with respect to the
kinetic variable, we obtain entropy solution to the considered
problem. Such a procedure is proposed in \cite{brenier} for Cauchy problems corresponding to equation \eqref{cl1-hom}. One of the aims of the paper is  to extend the transport-collapse
scheme \cite{brenier} for the initial value
problem for heterogeneous scalar conservation laws. 

The power of the method to be presented is in its
ability to transform nonlinear problem into linear. Linear scalar
conservation laws are easy to solve numerically since there are a
lot of robust numerical schemes available. The cost of that
"transformation" in practical computing is adding one more dimension
(see \eqref{kin1}).

Moreover, we shall use the transport-collapse techniques to construct the bounded function $u$ satisfying the following definition.

\begin{definition}
\label{def2}
We say that the function $u\in L^\infty(\R^+\times \Omega;[a,b])$ is a weak entropy admissible solution to \eqref{cl1-hom}, \eqref{ic1}, \eqref{bc1} if for every $k\in \R$ and every non-negative $\varphi\in C^1_c(\R^d_+;\R^+)$ it holds

\begin{align}
\label{semi+b}
&\int_{\R^d_+}\left(|u-k|_+ \pa_t \varphi+{\rm sgn}_+(u-k)(f(u)-f(k)) \, \nabla_{\mx} \varphi \right)d\mx dt+ \int_{\R^d}|u_0-k|_+\varphi(0,\cdot) d\mx \\& - \int_{0}^{b-a} \!\!\!\!\int\limits_{\substack{ \R^+\times \pa \Omega \\ \langle f'(\lambda-a),\vec{\nu}(\mx) \rangle < 0}} \!\!\!\!\varphi\,|\lambda-k-a|_+ \langle f'(\lambda-a),\vec{\nu}(\mx) \rangle \chi(\lambda, u_B(t,\mx)-a) d\gamma(\mx) dt d\lambda \geq 0, \nonumber
\end{align}and

\begin{align}
\label{semi-b}
&\int_{\R^d_+}\left(|u-k|_- \pa_t \varphi+{\rm sgn}_-(u-k)(f(u)-f(k)) \, \nabla_{\mx} \varphi \right)d\mx dt + \int_{\R^d}|u_0-k|_-\varphi(0,\cdot) d\mx \\& - \int_{a-b}^0\!\!\!\! \int\limits_{\substack{ \R^+\times \pa \Omega \\ \langle f'(\lambda+b),\vec{\nu}(\mx) \rangle < 0}}\!\!\!\! \varphi\,|\lambda-k+b|_- \langle f'(\lambda+b),\vec{\nu}(\mx) \rangle \chi(\lambda, u_B(t,\mx)-b) d\gamma(\mx) dt d\lambda \geq 0.\nonumber
\end{align} 
\end{definition} It is not difficult to see that if $u$ satisfies conditions of  Definition \ref{def2} then $u$ also satisfies Definition \ref{def-panov}. This will be proved in the last section.

Let us briefly comment Definition \ref{def2}. The first two terms on the left-hand sides of \eqref{semi+b} and \eqref{semi-b} are standard in the entropy admissibility concept (compare with Definition \ref{def-panov} and \eqref{a1-stand}) and they are related to the behaviour of the solution $u$ in the interior of $\Omega$ and on $t=0$. The final terms on the left-hand sides of \eqref{semi+b} and \eqref{semi-b} simply say that when the characteristics enter $\Omega$ (i.e. when the angle between the normal $\vec{\nu}$ and $f'(\lambda)$ is greater than $\pi/2$, i.e. when $\langle f'(\lambda_,\vec{\nu}(\mx) \rangle < 0 $) then we shall take the boundary data into account. Remark that we shifted a solution by $a$ in \eqref{semi+b} and by $b$ in \eqref{semi-b} since then $u-a\geq 0$ and $u-b\leq 0$, respectively, implying $\chi(\lambda, u(t,\mx)-a) \geq 0$ and $\chi(\lambda, u(t,\mx)-b) \leq 0$. This enabled us precise control of the behaviour of the solution at the boundary (see the last section).

Finally, let us remark that work in the field of numerical methods for conservation laws is
rather intensive. Most of the papers deal with Cauchy problems for
conservation laws (scalar conservation laws or systems; see e.g.
classical books \cite{HR, Lve} and references therein). As for
\eqref{cl1-hom}, \eqref{ic1}, \eqref{bc1}, there are not so many results
since the interest for this kind of problem has arisen relatively
recently. We mention \cite{AS, Vov} and references therein. For results in the case of systems, one can consult \cite{MS} where one can also find thorough overview of the subject.

The paper is organized as follows. In Section 2, we shall prove
convergence of the transport-collapse scheme for initial value
problems corresponding to \eqref{cl1}. In Section 3, we shall
introduce a transport-collapse type operator for \eqref{cl1-hom},
\eqref{ic1}, \eqref{bc1}, and the proof of its convergence toward
the entropy solution.



\section{Transport collapse scheme for the Cauchy problem for heterogeneous scalar conservation law}

The transport-collapse scheme is based on tracking of characteristics of conservation law \eqref{kin1}. In the homogeneous case (i.e. when the flux is independent of $(t,\mx)$), the characteristic have very simple form $x-f'(\lambda) t$ and it is significantly easier to analyse them than in the case when the flux is $(t,\mx)$-dependent. Thus, this section represents a non-trivial generalization of the method from \cite{brenier}. 

Let us first introduce assumptions on the flux $f$ from \eqref{cl1}. We assume $f=(f_1,\dots,f_d)\in C^2(\R_+^{d+1})$ and for some $b>a$
$$
f(t,\mx,a)=f(t,\mx,b)=0,
$$ and $a\leq u_0 \leq b$ for the initial condition $u_0$. Latter conditions provide the maximum principle for the entropy admissible solution to \eqref{cl1}, \eqref{ic1}. More precisely, the entropy admissible solution $u$ will stay bounded between $a$ and $b$. 

Let us now state properties of the function $\chi$.

\begin{proposition}\cite[page 1018]{brenier}
\label{chi}
It holds

\begin{itemize}

\item [a)]  $\forall u,v \in L^1(\R^d)$ such that $u \geq v \implies  \chi(\lambda,u) \geq \chi(\lambda,v)$;


\item [b)] $\forall u \in L^1(\R^d)$, $\forall g \in L^\infty(\R^d\times
\R)$, it holds \\ $ \iint \chi(\lambda,u) g(\mx,\lambda)d\mx
d\lambda=\int \left(\int_a^u g(\mx,\lambda)d\lambda \right)d\mx$;

In particular, if $g=G'_\lambda$ and $G(a)=0$, then $ \iint
\chi(\lambda,u) g(\mx,\lambda)d\mx d\lambda=\int G(\mx,u)d\mx$

\item [c)] $TV(u)=\int TV(\chi(\lambda,\cdot))d\lambda$.

\end{itemize}
\end{proposition}

The idea of the transport collapse scheme for the initial value
problem \eqref{cl1}, \eqref{ic1} is to solve problem \eqref{kin1},
\eqref{ickin} when we omit the right-hand side in \eqref{kin1}:
\begin{equation}
\label{TC} \pa_t h +\Div_{\mx,\lambda}[F(t,\mx,\lambda) h]=0, \ \
h|_{t=0}=\chi(\lambda,u_0(\mx)).
\end{equation} The solution of this equation is obtained via the method of characteristics. They are given by
\begin{equation}
\label{char}
\begin{cases}
 &\dot \mx= f'_\lambda, \ \ \mx|_{t=0}=\mx_0, \\
&\dot \lambda =-\sum\limits_{j=1}^d\pa_{x_j}f_j(t,\mx,\lambda), \ \ \lambda|_{t=0}=\lambda_0.
\end{cases}
\end{equation} For later purpose, we rewrite this system in the
integral form

\begin{equation}
\label{char1}
\begin{cases}
\mx&=\mx_0+\int_0^t f'_\lambda(t',\mx,\lambda)dt'\\
\lambda&=\lambda_0-\int_0^t\sum\limits_{j=1}^d\pa_{x_j}f_j(t',\mx,\lambda)dt'.
\end{cases}
\end{equation}

The solution to \eqref{TC} has the form
\begin{equation}
\label{s-char}
h(t,\mx,\lambda)=\chi(\lambda_0(t,\mx,\lambda),u_0(\mx_0(t,\mx,\lambda))).
\end{equation}

 To avoid
proliferation of symbols, denote for $\varphi=(\varphi_1,\dots,\varphi_d):\R^d\to \R^d$ the norms
\begin{equation}
\label{norm}
\begin{split}
&\|\nabla_{\mx} \varphi\|_{\infty}=\sup\limits_{1\leq k \leq d}\sup\limits_{\substack{x\in \R^d, \\ \!\! \Delta \mx >0}}
\frac{|\varphi_k(\mx+\Delta \mx)-\varphi_k(\mx)|}{|\Delta \mx |},\\
&\|\varphi\|_{\infty}=\sup\limits_{1\leq k \leq d}\sup\limits_{x\in \R^d}|\varphi_k(\mx)|,
\end{split}
\end{equation}where $|\;\cdot \;|$ denotes the Euclidean norm. We have the following properties of the characteristics.

\begin{proposition}
\label{charchar}
The characteristics $\mx_0=\mx_0(t,\mx,\lambda)$ and $\lambda_0 = \lambda_0(t,\mx,\lambda)$ satisfy the following continuity properties:
\begin{align}
\label{mx} |{\cal R}_{\mx}|&:=|\mx_0(t,\mx+\Delta
\mx,\lambda)-\mx_0(t,\mx,\lambda)|\\&\leq |\Delta \mx | \,\left( 1
+ \int_0^t \max\limits_{\lambda}\|\nabla_{\mx}
f'_\lambda(t',\cdot,\lambda)\|_{\infty} dt'\right). \nonumber
\\
 \label{lambda}
|{\cal R}_{\lambda}|&:=|\lambda_0(t,\mx+\Delta
\mx,\lambda)-\lambda_0(t,\mx,\lambda)|\\&\leq |\Delta \mx |\,
\int_0^t \max\limits_{\lambda}\| \nabla_{\mx}
\Div_{\mx}f(t',\cdot,\lambda)\|_{\infty} dt', \nonumber
\end{align} where the norms are given by \eqref{norm}.
\end{proposition}
\begin{proof} From \eqref{char1}, we have

\begin{align*}
\mx&=\mx_0(t,\mx,\lambda)+\int_0^t f'_\lambda(t',\mx,\lambda)dt'\\
\mx+\Delta \mx&=\mx_0(t,\mx+\Delta \mx,\lambda)+\int_0^t f'_\lambda(t',\mx+\Delta \mx,\lambda)dt'.
\end{align*} By subtracting those equations, we obtain:
\begin{align}
\label{ruj1}
&|\mx_0(t,\mx+\Delta \mx,\lambda)-\mx_0(t,\mx,\lambda)|\\
&\leq |\Delta \mx | +\int_0^t \max\limits_{\lambda}\|f'_\lambda(t',\mx+\Delta \mx,\lambda)-f'_\lambda(t',\mx,\lambda) \|_{\infty} dt'\nonumber\\
&\leq |\Delta \mx | +|\Delta \mx| \,\int_0^t
\max\limits_{\lambda}\|\nabla_{\mx} f'_\lambda(t',\cdot,\lambda)
\|_{\infty} dt'. \nonumber
\end{align} This proves \eqref{mx}. Inequality \eqref{lambda} is proved analogously. It holds
\begin{align*}
\lambda&=\lambda_0(t,\mx+\Delta \mx,\lambda)-\int_0^t\sum\limits_{j=1}^d\pa_{x_j}f_j(t',\mx+\Delta \mx,\lambda)dt',\\
\lambda&=\lambda_0(t,\mx,\lambda)-\int_0^t\sum\limits_{j=1}^d\pa_{x_j}f_j(t',\mx,\lambda)dt'
\end{align*} and it is enough to subtract the last two equalities,
and to follow the procedure from \eqref{ruj1}. \end{proof}

Let us now define the transport-collapse operator $T$.

\begin{definition}
The transport collapse operator $T(t)$ is defined for every $u\in
L^1(\R^d)$ by \vspace{-0.2cm}
\begin{equation}
\label{TC1} T(t)u(\mx)=\int
\chi(\lambda_0(t,\mx,\lambda),u(\mx_0(t,\mx,\lambda))) d\lambda.
\end{equation}
\end{definition}
It satisfies the following properties which are the same as the ones
from \cite[Proposition 1]{brenier}.

\begin{proposition}
\label{basic_prop} It holds for every $u,v \in L^1(\R^d)$

\begin{itemize}

\item [a)] $u \leq v$ a.e. implies $T(t)u \leq T(t)v$ a.e;

\item [b)] $\int T(t)u(\mx)d\mx = \int u(\mx) d\mx$;

\item [c)] the operator $T(t)$ is non-expansive
$$
\|T(t)u-T(t)v\|_{L^1(\R^d)} \leq \|u-v\|_{L^1(\R^d)},
$$and, in particular, $\|T(t)u\|_{L^1(\R^d)} \leq \|u\|_{L^1(\R^d)}$;

\item [d)] $TV(T(t)u)\leq (1+C_1 t)\,TV(u)+t C_2$, where $TV$ is the total
variation and $C_1$ and $C_2$ are appropriate constants depending on
the $C^2$-bounds of the flux $f$;

\item [e)] $\|T(t)u-u\|_{L^1(\R^d)} \leq C_2 TV(u)t+t C_1$ for the constants $C_1$ and $C_2$ from the previous item;


\end{itemize}

\end{proposition}

\begin{proof}

Item a) directly follows from the definition of the transport
collapse operator $T(t)$.

As for the item b), for every fixed $t>0$, denote by $Z=(t,\mx,\lambda)$ characteristics from
\eqref{char}. Notice that, since $\Div_{(\mx,\lambda)} F=0$, it
holds
\begin{equation}
\label{1} \Big|\det\frac{\pa Z(t,\mx_0,\lambda_0)}{\pa
(\mx_0,\lambda_0)} \Big|=1.
\end{equation} Therefore, according to Proposition \ref{chi},
\begin{align}
\label{11}
&\int T(t)u(\mx) d\mx=\int_{\R^{d+1}}\chi(\lambda_0(t,\mx,\lambda),u(\mx_0(t,\mx,\lambda)) d\mx d\lambda\\
&={\mx_0(t,\mx,\lambda)=\my  \choose \lambda_0(t,\mx,\lambda)=\eta
}=\int_{\R^{d+1}}\chi(\eta,u(\my)) \Big|\det\frac{\pa
Z(t,\mx_0,\lambda_0)}{\pa (\mx_0,\lambda_0)} \Big| d\my d\eta=\int
u(\my)d\my.
\nonumber
\end{align}

Item c) now follows from a) and b) according to the Crandall-Tartar
lemma about non-expansive order preserving mappings \cite[Proposition 3.1]{7}.

Let us now prove item d). We have
\begin{align*}
&\int\limits_{\R^d} |T(t)u(\mx+\Delta
\mx)-T(t)u(\mx)|d\mx\\&=\!\int\limits_{\R^{d}}\!|\!\int\limits_{\R}\!\chi(\lambda_0(t,\mx\!+\!\Delta
\mx,\lambda),u(\mx_0(t,\mx\!+\!\Delta
\mx,\lambda)))\!-\!\chi(\lambda_0(\mx_0(t,\mx,\lambda),u(\mx_0(t,\mx,\lambda)))d\lambda| d\mx \\
&\leq \!
\!\int\limits_{\R^{d+1}}\!|\!\chi(\lambda_0(t,\mx\!+\!\Delta
\mx,\lambda),u(\mx_0(t,\mx\!+\!\Delta
\mx,\lambda)))\!-\!\chi(\lambda_0(\mx_0(t,\mx,\lambda),u(\mx_0(t,\mx,\lambda)))|
d\mx d\lambda
\end{align*} We next write $\mx_0(t,\mx+\Delta
\mx,\lambda)=\mx_0(t,\mx,\lambda)+{\cal R}_{\mx}(t,\mx,\lambda)$ and
$\lambda_0(t,\mx+\Delta \mx,\lambda)=\lambda_0(t,\mx,\lambda)+{\cal
R}_{\lambda}(t,\mx,\lambda)$, where ${\cal R}_{\mx}$ and ${\cal
R}_{\lambda}$ are estimated in \eqref{mx}, and introduce the change
of variables $\mx_0(t,\mx,\lambda)=\my$,
$\lambda_0(t,\mx,\lambda)=\eta$ (keep in mind \eqref{1}). We obtain

\begin{align*}
&\int_{\R^d} |T(t)u(\mx+\Delta \mx)-T(t)u(\mx)|d\mx\\
&\leq \int_{\R^{d+1}}|\chi(\eta+{\cal R}_\lambda,u(\my+{\cal R}_\mx))-\chi(\eta,u(\my))| d\my d\eta\\
&\leq \int_{\R^{d+1}}|\chi(\eta+{\cal R}_\lambda,u(\my+{\cal R}_\mx))-\chi(\eta,u(\my+{\cal R}_\mx))| d\my d\eta\\
& \quad + \int_{\R^{d+1}}|\chi(\eta,u(\my+{\cal R}_\mx))-\chi(\eta,u(\my))| d\my d\eta\\
& \leq \|{\cal R}_\lambda\|_{\infty} \, TV(\chi)+ \|{\cal
R}_\mx\|_{\infty}\int_{\R}TV(\chi(\eta,u(\cdot)))d\eta = 4\|{\cal
R}_\lambda\|_{\infty}+\|{\cal R}_\mx\|_{\infty} TV(u),
\end{align*} since the characteristics are of $C^1$-class, $TV(\chi)=4$,
and since Proposition \ref{chi}, item c) holds. Remark that in the case when $u\geq 0$ we actually have $\chi(\lambda,u)={\rm sgn}_+(u-\lambda)$ and in that case $TV(\chi)=1$. Having in mind Proposition
\ref{charchar}, we conclude the proof of d). We remark that
$$
C_1=4\max\limits_{t,\lambda}\| \nabla_{\mx} f(t,\cdot,\lambda)
\|_\infty, \ \ C_2= \max\limits_{t,\lambda}\|\nabla_{\mx}\Div_{\mx}
f(t,\cdot,\lambda)\|_{\infty}.
$$

It remains to prove item e). Using \eqref{char1}, as in to the proof
of item d), we have
\begin{align*}
&\|T(t)u-u\|_{L^1(\R^d)} \leq \int_{\R^{d+1}}
|\chi(\lambda_0(t,\mx,\lambda),u(x_0(t,\mx,\lambda)))-\chi(\lambda,u(\mx))|d\mx
d\lambda
\\ & = \int_{\R^{d+1}} |\chi(\lambda+{\cal R}_\lambda,u(\mx+{\cal R}_{\mx}))-\chi(\lambda,u(\mx))|d\lambda d\mx
\\ & \leq  C_1\,t\, TV(u)+C_2 \, t,
\end{align*} which immediately gives e).
\end{proof}


We also need the following result.

\begin{proposition}
\label{prop_T}
For any smooth positive test function $\varphi$, any $u\in L^1(\R)$ such that $a\leq u\leq b$,
and convex Lipschitz function $V:\R\to \R$, we have
\begin{align}
\label{1220} \int(V(T(t)u)-V(u))(\mx) \varphi(\mx) d\mx &\leq
\int_0^t \int B_V(t',\mx,u(\mx)) \nabla \varphi d\mx dt'\\&+
\int_0^t\int \int_a^u \Div_{\mx}f(t,\mx,\lambda)
V''(\lambda)d\lambda dt'+ o(t), \ \ t\to 0 \nonumber
\end{align} where $B_V(t,\mx,u)=\int_a^u
f'_\lambda(t,\mx,\lambda)V'(\lambda)d\lambda$, and $o(t)$ depends only on the $L^\infty$-bound of $u$.
\end{proposition}

\begin{proof}
Remark first that for any fixed $(t,\mx)$, from the definition of
the function $\chi$, it follows for any $C^1$-function $G$
\begin{equation}
\label{pm} \int G'(\lambda)\chi(\lambda_0(t,\mx,\lambda), u(\mx_0(t,\mx,\lambda)))
d\lambda=\sum\limits_{k=0}^{2p} (-1)^k G(\omega_k)-G(0),
\end{equation}where the increasing sequence $(\omega_k)$,
$k=0,\dots,2p$, belongs to the set $N\!ul=\{\lambda \in [a,b]:
\lambda_0(t,\mx,\lambda)=u(\mx_0(t,\mx,\lambda))\}$ (since the
entropy solution to \eqref{cl1}, \eqref{ic1} takes values in the
interval $(a,b)$). Indeed, for almost every $(t,\mx)\in \R^+\times \R$, it holds
$
N\!ul=\{\lambda \in [a,b]:
\lambda_0(t,\mx,\lambda)=u(\mx_0(t,\mx,\lambda))\}=\{\omega_0,\dots,\omega_{2p}\}.
$  In the intervals $(\omega_k,\omega_{k+1})$ and $(\omega_{k+1},\omega_{k})$ the function $ \lambda \mapsto \lambda_0(t,\mx,\lambda)-u(\mx_0(t,\mx,\lambda))$ has different signs, and we can assume
\begin{equation*} 
\begin{split}
\lambda_0(t,\mx,\lambda)>u(\mx_0(t,\mx,\lambda)), \ \ \lambda \in (\omega_{2s},\omega_{2s+1}), \; s=0,\dots,p;\\
\lambda_0(t,\mx,\lambda)<u(\mx_0(t,\mx,\lambda)), \ \ \lambda \in (\omega_{2m+1},\omega_{2m}), \; m=0,\dots,p.
\end{split}
\end{equation*}

To be more concise, recall that we assumed $a=0$ and $b>0$. According to definition of the kinetic function $\chi$, we see that $\chi(\lambda_0(t,\mx,\lambda),\mx_0(t,\mx,\lambda))=1$ for $\lambda \in (\omega_{2m+1},\omega_{2m})$,  $m=0,\dots,p$, and $\chi(\lambda_0(t,\mx,\lambda),\mx_0(t,\mx,\lambda))=0$ for $\lambda \in (\omega_{2s},\omega_{2s+1})$,  $s=0,\dots,p$. From here, \eqref{pm} immediately follows.

Remark that the set has odd cardinality since the
multivalued solution is obtained by continuous transformation from
the graph of initial value \cite[page 1016]{brenier}. Moreover, due
to the mean value theorem, the following relation holds for any
convex function $V$ (see e.g. \cite[p. 51]{Dieux}):

\begin{equation}
\label{conv} V(\sum\limits_{k=0}^{2p} (-1)^k \omega_k)\leq
\sum\limits_{k=0}^{2p} (-1)^k V(\omega_k).
\end{equation}

From \eqref{pm} and \eqref{conv}, it follows
\begin{align}
\label{**}
&V(T(t)u(\mx))=V(\int
\chi(\lambda_0(t,\mx,\lambda),u(t,\mx_0(t,\mx,\lambda)))d\lambda) \overset{\eqref{pm}}{=} V(\sum\limits_{k=0}^{2p}
(-1)^k \omega_k)\\
&\overset{\eqref{conv}}{\leq} \sum\limits_{k=0}^{2p} (-1)^k V(\omega_k)\overset{\eqref{pm}}{=} \int
V'(\lambda)\chi(\lambda_0(t,\mx,\lambda), u(\mx_0(t,\mx,\lambda))) d\lambda+V(0).
\nonumber
\end{align} Moreover, it holds
\begin{equation}
\label{***}
V(u(\mx))=\int V'(\lambda) \chi(\lambda,u(\mx))d\lambda-V(0).
\end{equation} Subtracting \eqref{***} from \eqref{**}, we reach to
\begin{equation}
\label{****}
V(T(t)u(\mx))-V(u(\mx))\leq \int
V'(\lambda)\left(\chi(\lambda_0(t,\mx,\lambda), u(\mx_0(t,\mx,\lambda)))-\chi(\lambda,u(\mx))\right) d\lambda.
\end{equation}

We have from here 
\begin{align}
\label{ruj2}
&\int\left( V(T(t)u(\mx))-V(u(\mx))\right) \varphi(\mx) d\mx \\
\nonumber &\leq \iint \left( V'(\lambda)
\chi(\lambda_0(t,\mx,\lambda),u(\mx_0(t,\mx,\lambda))) - V'(\lambda)
\chi(\lambda,u(\mx)) \right) \varphi(\mx) d\mx d\lambda\\
\label{ruj3}&=\! \iint  V'(\lambda_0(t,\mx,\lambda))
\chi(\lambda_0(t,\mx,\lambda),u(\mx_0(t,\mx,\lambda)))
(\varphi(\mx)\!-\!\varphi(\mx_0(t,\mx,\lambda)))d\mx
d\lambda\\
\label{ruj5} & +\! \iint
(V'(\lambda)\!-\!V'(\lambda_0(t,\mx,\lambda)))
\chi(\lambda_0(t,\mx,\lambda),u(\mx_0(t,\mx,\lambda)))
\varphi(\mx)d\mx d\lambda\\
\nonumber &+ \Big(\iint  V'(\lambda_0(t,\mx,\lambda))
\chi(\lambda_0(t,\mx,\lambda),u(\mx_0(t,\mx,\lambda))) \,
\varphi(\mx_0(t,\mx,\lambda))d\mx d\lambda
\\&\qquad\qquad\qquad\qquad\qquad\qquad\qquad\qquad-\iint V'(\lambda) \chi(\lambda,u(\mx))
\varphi(\mx) d\mx d\lambda\Big). \label{ruj7}
\end{align} The two terms from \eqref{ruj7} cancel according to
\eqref{1}.  Indeed, using the change of variables from \eqref{1}, with the notation from \eqref{11}, we conclude  
\begin{align*}
&\iint  V'(\lambda_0(t,\mx,\lambda))
\chi(\lambda_0(t,\mx,\lambda),u(\mx_0(t,\mx,\lambda))) \,
\varphi(\mx_0(t,\mx,\lambda))d\mx d\lambda\\&=\iint V'(\eta) \chi(\eta,u(\my))
\varphi(\my) d\my d\eta.
\end{align*} 

Let us now consider the term from \eqref{ruj3}. Using the Taylor formula

\begin{align}
\label{ruj9} &\iint  V'(\lambda_0(t,\mx,\lambda))
\chi(\lambda_0(t,\mx,\lambda),u(\mx_0(t,\mx,\lambda)))
(\varphi(\mx_0(t,\mx,\lambda))-\varphi(\mx)) d\mx d\lambda\\
&= \iint  V'(\lambda_0(t,\mx,\lambda))
\chi(\lambda_0(t,\mx,\lambda),u(\mx_0(t,\mx,\lambda)))
(\mx_0(t,\mx,\lambda)-\mx)\cdot \nabla
\varphi(\mx_0(t,\mx,\lambda)) d\mx d\lambda\nonumber\\
&+ \frac{1}{2}\iint  V'(\lambda_0(t,\mx,\lambda))
\chi(\lambda_0(t,\mx,\lambda),u(\mx_0(t,\mx,\lambda))) D^2
\varphi(\tilde{\mx}) \,(\mx_0(t,\mx,\lambda)-\mx)^2 d\mx d\lambda,
\nonumber
\end{align} where $\tilde{\mx}$ is a point belonging to a neighbourhood of $x_0(t,\mx,\lambda)$. To further estimate the  latter term, we expand the
function $f'_\lambda(t',\mx,\lambda)$ into the Taylor expansion
around $\mx_0$ and take \eqref{char1} into account:

\begin{align}
\label{ruj13} \mx&-\mx_0(t,\mx,\lambda)= \int_0^t
f'_\lambda(t',\mx,\lambda)dt'=\int_0^t
f'_\lambda(t',\mx_0(t,\mx,\lambda),\lambda)dt'\\&+(\mx_0(t,\mx,\lambda)-\mx)\cdot\int_0^t
\nabla_x f'_\lambda(t',\tilde{\mx},\lambda)dt'=\int_0^t
f'_\lambda(t',\mx_0(t,\mx,\lambda),\lambda)dt'+{\cal O}(t^2),
\nonumber
\end{align} since clearly $\mx_0(t,\mx,\lambda)-\mx={\cal O}(t)$ and $\int_0^t
f'_\lambda(t',\mx,\lambda)dt'={\cal O}(t)$. Inserting this into
\eqref{ruj9} and applying the change of variables from \eqref{1}, we
conclude using item b) from \eqref{chi}:
\begin{align}
\label{ruj11} \iint  V'(\lambda_0(t,\mx,\lambda)) \chi(\lambda,u(\mx))
(\varphi(\mx_0(t,\mx,\lambda))-\varphi(\mx)) d\mx
d\lambda&\\=\int_0^t B_V(t',\mx,u(\mx)) \nabla \varphi d\mx dt'+
{\cal O}(t^2).& \nonumber
\end{align}

To deal with the remaining term from \eqref{ruj5}, we shall expand
the function $V'$ into the Taylor series around $\lambda_0$.  We
have

\begin{align}
\label{ruj19} &\iint  (V'(\lambda)-V'(\lambda_0))
\chi(\lambda_0(t,\mx,\lambda),u(\mx_0(t,\mx,\lambda))) \,
\varphi(\mx)d\mx d\lambda\\
&=\iint
V''(\lambda_0(t,\mx,\lambda))(\lambda-\lambda_0((t,\mx,\lambda)))
\chi(\lambda_0(t,\mx,\lambda),u(\mx_0(t,\mx,\lambda))) \,
\varphi(\mx)d\mx d\lambda \nonumber\\&+{\cal
O}(\|\lambda-\lambda_0((t,\mx,\lambda)) \|^2_{L^1({\rm
supp}(\varphi)\times (a,b))}) \nonumber
\end{align} Applying the procedure as in \eqref{ruj13} and having in mind \eqref{char1}, we reach to
the estimate

\begin{equation}
\label{ruj15} \lambda_0(t,\mx,\lambda)-\lambda= -\int_0^t
\sum\limits_{j=1}^d f_j(t',\mx_0(t,\mx,\lambda),\lambda)dt'+{\cal
O}(t^2).
\end{equation} If we notice that $\|\lambda-\lambda_0((t,\mx,\lambda)) \|^2_{L^1({\rm
supp}(\varphi)\times (a,b))}={\cal O}(t^2)$, from \eqref{ruj19} and
\eqref{ruj15}, we conclude

\begin{align}
\label{ruj21} &\iint  (V'(\lambda)-V'(\lambda_0))
\chi(\lambda_0(t,\mx,\lambda),u(\mx_0(t,\mx,\lambda))) \,
\varphi(\mx_0(t,\mx,\lambda))d\mx d\lambda\\&=-\int_0^t \int
\int_a^u \sum\limits_{j=1}^d f_j(t',\mx_0(t,\mx,\lambda),\lambda)
\, \varphi(\mx) d\lambda d\mx dt'+{\cal O}(t^2) \nonumber
\end{align} Combining \eqref{ruj2}, \eqref{ruj11}, and
\eqref{ruj21}, we conclude the theorem. \end{proof}

A consequence of Proposition \ref{basic_prop} and Proposition
\ref{prop_T} is the following theorem:

\begin{theorem}
\label{mainthm}
Denote
\begin{equation}
\label{TCoper} S_n(t)u=(1-\alpha)T(\frac{t}{n})^ku\,+\,\alpha
T(\frac{t}{n})^{k+1}u,
\end{equation}where
\begin{equation}
\label{alpha}
t=\frac{(k+\alpha)}{n}, \ \ k\in \N, \ \ \alpha\in [0,1).
\end{equation}
For each initial value $u_0\in L^1(\R^d)$ such that
$a\leq u_0 \leq b$, the unique entropy solution of \eqref{cl1},
\eqref{ic1} at time $t$ is given by the formula
$$
u(t,\cdot)=L^1- \lim\limits_{n\to \infty} S_n(t)u_0.
$$
\end{theorem}

\begin{proof}
First, fix an arbitrary $t>0$. Consider the sequence of functions
$u_n(t,\cdot)=S_n(t) u$. We aim to prove that the sequence
$(u_n(t,\cdot))$ is strongly precompact in $L^1(\R^d)$. To this end,
we shall use the Kolmogorov criterion stating  that a functional
sequence bounded in $L^1(\R^d)$ is strongly precompact in
$L^1(\R^d)$ if it is uniformly $L^1(\R^d)$ continuous. In other
words, we need to prove that
\begin{itemize}

\item [a)] $\|u_n(t,\cdot)\|_{L^1(\R^d)} \leq C$ for every $n\in \N$ and some constant $C$;

\item [b)] for any relatively compact $K\subset\subset \R^d$, any $\eps>0$,
there exists $\Delta x>0$ such that $\|u_n(t,\mx+\Delta \mx)-
u_n(t,\mx)\|_{L^1(\R^d)} \leq \eps$.

\end{itemize}

Item a) follows from Proposition \ref{basic_prop}, item c) (we take $v=0$ there).

As for the item b), we shall use (recursively) property d) from Proposition
\ref{basic_prop}. Taking into account definition of the total
variation and form of the sequence $(u_n(t,\cdot))$, simple
calculations show that (with the notations from Proposition
\ref{basic_prop})
\begin{align*}
&TV(u_n)\leq (1+C_1 \frac{t}{n})TV(T(\frac{t}{n})^{n-1}(u_0))+C_2\frac{t}{n}\leq \dots\\ &\leq \left( 1 + C_1\frac{t}{n}\right)^nTV(u)+\frac{C_2 t}{n} \sum\limits_{k=0}^n \left(1+C_1\frac{t}{n}\right)^k \leq exp(C_1 t)TV(u_0)+2C_2(t+t^2).
\end{align*}  Since $u_0\in BV(\R^d)$, this immediately implies
$L^1$-equicontinuity of the sequence $(u_n(t,\cdot))$. This means
that for every fixed $t>0$, we can choose a strongly converging
subsequence (not relabelled) $(u_n(t,\cdot))$ of the sequence
$(u_n(t,\cdot))$. By taking a dense countable subset $E\subset \R^+$,
we can choose the same converging subsequence $(u_n(t,\cdot))$ for
every $t\in E$.

Now, by the continuity property given in item e) from Proposition
\ref{basic_prop}, we conclude that the subsequence $(u_n(t,\cdot))$
strongly converges in $C([0,T];L^1(\R^d)$ for every $T\in \R^+$
toward a function $u\in C([0,T];L^1(\R^d)$.

Now, we need to check that $u$ satisfies the entropy admissibility
conditions. First, notice that for every $t$, as $n\to \infty$, it
holds that $\alpha \to 0$. Thus, it is enough to notice that the main
part of the transport-collapse operator given by $T(\frac{t}{n})^ku
\to u$ as $n\to \infty$ along the previously chosen subsequence and
to consider

\begin{align}
\label{***}
\int_{\R^d}(V(T(\frac{t}{n})^ku)&-V(u)) \varphi(\mx) d\mx
=\sum\limits_{j=0}^{k-1}\int_{\R^d}(V(T(\frac{t}{n})^{j+1} u)-V(T(\frac{1}{n})^{j} u))\varphi(\mx) d\mx\\
&\overset{\eqref{1220}}{\leq} \sum\limits_{j=0}^{k-1}
\int_{jt/n}^{(j+1)t/n} \int_{\R^d} B_V(t',\mx,T(\frac{t}{n})^j
u(\mx)) \nabla \varphi d\mx dt'+{\cal O}(t/n).
\nonumber
\end{align} Now, we simply let $n\to \infty$ and keep in mind arbitrariness of
$t$ to infer that the function $u$ satisfies the entropy
admissibility conditions from Definition \ref{def1}, a).

Remark also that this implies convergence of the entire sequence
given by \eqref{TCoper} due to uniqueness of entropy solutions to
\eqref{cl1}, \eqref{ic1}. \end{proof}

\section{Boundary value problem}

In this section, we shall consider boundary value problem for homogeneous scalar conservation law \eqref{cl1-hom} on the domain $\Omega$,
which is a bounded simply connected open smooth subset of $\R^d$. 


In order to simplify the presentation, we shall assume that $a=0$ in \eqref{bndass}, i.e. that the solution to the considered problem is non-negative. In particular, this implies that the kinetic function $\chi$ corresponding to such a solution satisfies
\begin{equation}
\label{ass-cns}
\chi(\lambda,u)={\rm sgn}_+(u-\lambda) \geq 0.
\end{equation}

First, notice that the kinetic formulation from Theorem 2 still
holds in the interior of $\R^+\times\Omega$. However, we cannot use the method of
characteristics from the previous section directly since the characteristics entering the
boundary determine the value at the boundary. Nevertheless, since we are
re-iterating the procedure after a short period of time (see
\eqref{TCoper}), we can modify the transport collapse scheme so that we take into account the boundary data.

Accordingly, recall that the kinetic reformulation for \eqref{cl1-hom} has the form:
\begin{equation}
\label{hom-kinetic}
\pa_t \chi(\lambda,u)+f'(\lambda) \Div_{\mx} \chi(\lambda,u)=\pa_\lambda m_+(t,\mx,\lambda)
\end{equation} where $m_+$ is a non-negative measure. Assume that $\Omega$ is an open set such that for some $\sigma\in (0,1)$, no two outer normals from $\pa \Omega$ do not intersect in the set 
\begin{align*}
&\Omega_{\sigma}=\{\mx \in \R^d: \ \ dist(\mx,\Omega)<\sigma \} \ \ \text{i.e. in the set}\\
&\Omega^{\sigma}=\Omega_{\sigma} \backslash \Omega
\end{align*} (i.e. we assume that $\Omega$ has finite curvature). In order to augment \eqref{hom-kinetic} (with neglected right-hand side) with appropriate initial data, denote by $\vec{\nu}(\mx)$, $\mx \in \Omega_\sigma\setminus \Omega$ the unit outer normal on $\pa\Omega$ passing trough the point $\mx$.  We then extend the boundary data $u_B(t,\mx)$ for every fixed $t\geq 0$ along the normals $\vec{\nu}(\mx)$ in the set $\Omega_{\sigma}$. 

More precisely, we set for $\mx\in \Omega^\sigma=\Omega_{\sigma} \backslash \Omega$ (slightly abusing the notation)
\begin{equation}
\label{sigma}
u_B(t,\mx)=u_B(t,\mx_0), \ \ \text{for $\mx_0\in \pa \Omega$ such that  } \vec{\nu}(\mx_0)=\vec{\nu}(\mx).
\end{equation} Finally, introduce the function
\begin{equation}
\label{w}
w_{u(t,\cdot)}(\mx)=\begin{cases}
0, & \mx\notin \Omega_\sigma\\
u(t,\mx), & \mx\in \Omega\\
u_B(t,\mx), & \mx\in \Omega_\sigma\setminus \Omega=\Omega^\sigma,
\end{cases}
\end{equation}which is actually the extension of $u$ along the normals $\vec{\nu}$. If the function $u$ does not depend on $t$, then we put $u(\mx)$ instead of $u(t,\mx)$, and $u_B(0,\mx)$ instead of $u_B(t,\mx)$ on the right-hand side of \eqref{w}. Remark that we can rewrite the function $w_{u(t,\cdot)}(\mx)$ in the form
$$
w_{u(t,\cdot)}(\mx)=u(t,\mx) \kappa_{\Omega}(\mx)+u_B(t,\mx)\kappa_{\Omega^\sigma}(\mx),
$$where $\kappa_A$ is the characteristic function of the set $A$.

Now, we are ready to introduce a modification of the transport collapse scheme from the previous section. Fix $t>0$ and $n\in \N$. We neglect the right-hand side of \eqref{hom-kinetic} and, on the first step, we augment it with $\chi(\lambda,w_{u_0}(\mx))$ as the initial data.

\begin{align} 
\label{hom-TC}
\pa_t h+f'(\lambda) \Div_{\mx} h&=0, \\ 
h|_{t=0}&=\chi(\lambda,w_{u_0}(\mx)).
\label{id-hom}
\end{align} The solution to \eqref{hom-TC} is given by $h(t,\mx,\lambda)=\chi(\lambda, \omega_{u_0}(\mx-f'(\lambda) t))$ (since the characteristics of the equation have quite simple form; see \cite{brenier}). We construct the approximate solution $u_n$ to \eqref{cl1-hom}, \eqref{ic1}, \eqref{bc1} by the following procedure:

\begin{itemize}

\item \begin{equation}
\label{aver}
u_n(t',\mx)=T(t'/n)(w_{u_0}(\mx)):=\int_0^b \chi(\lambda, \omega_{u_0}(\mx-f'(\lambda) t')) d\lambda, \ \ t'\in (0,t/n].
\end{equation}

\item  For $k=1,\dots,n-1$, we take
\begin{equation}
\label{aver1}
u_n(kt/n+t',\mx)=\int_0^b \chi(\lambda, \omega_{u_n(kt/n,\cdot)}(\mx-f'(\lambda) t')) d\lambda, \ \ t'\in (0,t/n].
\end{equation}

\end{itemize} Remark that here, we have actually applied the transport collapse operator. 
Roughly speaking, the approximate solution in $[0,t]\times \Omega$ is given by the transport-collapse operator, while in $[0,t]\times \Omega^C$ the sequence $(u_n)$ is equal to the boundary data extended along the normals on $\pa \Omega$.

We shall show that the sequence $(u_n)$ strongly converges in $L^1([0,t]\times \Omega)$ along a subsequence toward a function $u$ which represents the solution to \eqref{cl1-hom}, \eqref{ic1}, \eqref{bc1} in the sense of Definition \ref{def2}. In order to prove the later fact, we shall use the kinetic formulation similar to \cite{pan_tams}.  We introduce the following definition.

\begin{definition}
\label{def-kin}
We say that the non-negative function $p_+ \in L^\infty(\R^+\times \Omega \times \R)$ is the kinetic super-solution to \eqref{cl1-hom}, \eqref{ic1}, \eqref{bc1} if it satisfies the following equation for any $\varphi \in C^1_c(\R_+^d)$ and $\rho \in C_c^1(\R)$ (recall that we assumed $a=0$ in \eqref{bndass}):

\begin{align}
\label{kin+}
&\int_{\R^d_+\times \R} \rho(k)p_+(t,\mx,k) \left( \pa_t \varphi +  f'(k) \, \nabla_{\mx} \varphi \right) d\mx dt dk \\&- \int_{\R}\int_0^b \int\limits_{\substack{ \R^+\times \pa \Omega \\ \langle f'(\lambda),\vec{\nu}(\mx) \rangle < 0}} \!\!\!\!\rho(k) \varphi\,{\rm sgn}_+(\lambda-k) \langle f'(\lambda),\vec{\nu}(\mx) \rangle \chi(\lambda, u_B(t,\mx)) d\gamma(\mx) dt d\lambda dk \nonumber\\&=- \int_{\R}\int_{\R^+\times \Omega} \varphi(t,\mx) \rho'(k) dm^+(t,\mx,k),
\nonumber
\end{align}for a non-negative measure $m^+$.

We say that the non-positive function $p_- \in L^\infty(\R^+\times \Omega \times \R)$ is the kinetic sub-solution to \eqref{cl1-hom}, \eqref{ic1}, \eqref{bc1} if it satisfies the following equation for any $\varphi \in C^1_c(\R_+^d)$ and $\rho \in C_c^1(\R)$:

\begin{align}
\label{kin-}
&\int_{\R^d_+\times \R} \rho(k)p_-(t,\mx,k) \left( \pa_t \varphi +  f'(k) \, \nabla_{\mx} \varphi \right) d\mx dt dk \\&- \int_{\R}\int_{-b}^0 \int\limits_{\substack{ \R^+\times \pa \Omega \\ \langle f'(\lambda),\vec{\nu}(\mx) \rangle < 0}} \!\!\!\!\rho(k) \varphi\,{\rm sgn}_+(\lambda-k+b) \langle f'(\lambda+b),\vec{\nu}(\mx) \rangle \times \nonumber\\&\qquad\qquad\qquad\qquad\qquad\qquad\qquad \qquad \times \chi(\lambda, u_B(t,\mx)-b) d\gamma(\mx) dt d\lambda dk \nonumber\\&= -\int_{\R}\int_{\R^+\times \Omega} \varphi(t,\mx) \rho'(k) dm^-(t,\mx,k),
\nonumber
\end{align}for a non-negative measure $m^-$.

The function $p\in L^\infty(\R^+\times \Omega \times \R)$ is the kinetic solution if it is kinetic super-solution and $(1-p)$ is kinetic sub-solution.

\end{definition}

The following theorem holds.

\begin{theorem}
\label{thm-bvp}
There exists the kinetic solution to  \eqref{cl1-hom}, \eqref{ic1}, \eqref{bc1} in the sense of Definition \ref{def-kin}.
\end{theorem}
\begin{proof}
We shall prove that the sequences of functions $({\rm sgn}_{\pm}(u_n-k))$ (weakly) converge toward the kinetic super and sub solutions for the sequence $(u_n)$ defined by \eqref{aver} and \eqref{aver1}. Since ${\rm sgn}_{+}(u_n-k)=1-{\rm sgn}_{-}(u_n-k)$ the kinetic super-solution will be the kinetic solution at the same time. 

First, remark that for every $n\in \N$, value of the function $u_n(t,\mx)$ for $\mx\in \Omega$ is given by the transport collapse operator.

Therefore, it is enough to consider behaviour of $V(T(t)v)-V(v)$ for a convex function $V$ whose special form will be chosen later, and for the function $v\geq 0$ playing the role of $u_n(t_s,\cdot)$ (recall that we have assumed that $a=0$ implying that our sequence of approximate solutions is non-negative) such that $v(t,\mx)=u_B(t,\mx)$, $\mx \in \Omega_\sigma\backslash \Omega$. Accordingly, denote by $\vec{\nu}_{\Omega-f'(\lambda)t}(\mx) $ the unit vector to $\pa (\Omega-f'(\lambda)t)$, assume that we fixed $t<\sigma$ for $\sigma$ given in \eqref{sigma}, and consider for $\varphi \in C^2_c(\Omega_\sigma)$:

\begin{align}
\label{hom4}
&\int_{\Omega} \left(V(T(t)v)(\mx)-V(v)(\mx) \right) \varphi(\mx) d\mx \\ \nonumber &\overset{\eqref{****}}{\leq} \int_{\Omega} \int_{0}^b V'(\lambda) \left(\chi(\lambda,v(\mx-f'(\lambda)t))-\chi(\lambda,v(\mx)) \right) \varphi(\mx) d\lambda d\mx= \Big(\my=\mx-f'(\lambda)t \Big)
\\& =\int_{0}^b \int_{\Omega-f'(\lambda)t}  V'(\lambda) \chi(\lambda,v(\my)) \varphi(\my+f'(\lambda) t)  d\my d\lambda-\int_{0}^b \int_{\Omega}  V'(\lambda) \chi(\lambda,v(\mx)) \varphi(\mx)  d\mx d\lambda \nonumber \\
&=\int_0^b \int\limits_{\substack{(\Omega-f'(\lambda)t)\backslash \Omega\\\langle f'(\lambda), \vec{\nu}_{\Omega-f'(\lambda)t}(\mx)\rangle< 0}} V'(\lambda) \chi(\lambda,u_B(0,\mx)) \varphi(\mx+t f'(\lambda))  d\mx d\lambda \nonumber\\
&-\int_0^b \int\limits_{\substack{\Omega\backslash (\Omega-f'(\lambda)t)\\\langle f'(\lambda), \vec{\nu}_{\Omega-f'(\lambda)t}(\mx)\rangle \geq 0}} V'(\lambda) \chi(\lambda,v(\mx)) \varphi(\mx)  d\mx d\lambda \nonumber\\
&+\int_0^b \int_{(\Omega-f'(\lambda)t)\cap \Omega} V'(\lambda) \chi(\lambda,v(\mx)) (\varphi(\mx+t f'(\lambda))-\varphi(\mx))  d\mx d\lambda.
\nonumber
\end{align} Now, since $v\geq 0$ it will also be $\chi(\lambda,v) \geq 0$ (see \eqref{ass-cns}. Next, for a fixed $k\in \R$, choose $V(\lambda)=V_+(\lambda)=|\lambda-k|_+$ in \eqref{hom4}. 
After expanding the function $\varphi$ in the Taylor expansion around $\mx$ and taking into account that as $t\to 0$:
$$ 
\frac{1}{t}\int\limits_{\substack{(\Omega-f'(\lambda)t)\backslash \Omega\\\langle f'(\lambda), \vec{\nu}(\mx)\rangle < 0}} g(\mx) d\mx\to -\!\!\!\!\!\!\!\!\!\!\int\limits_{\substack{\pa \Omega \\ \langle f'(\lambda), \vec{\nu}(\mx)\rangle < 0}} g(\mx)  \langle f'(\lambda), \vec{\nu}(\mx)\rangle d\gamma(\mx) \geq 0, 
$$ we get (keep in mind that $V_+' \geq 0$)

\begin{align}
\label{hom5}
&\int_{\Omega} \left(V_+(T(t)v)(\mx)-V_+(v)(\mx) \right) \varphi(\mx) d\mx 
\\
&\qquad\qquad\leq -t\int_0^b \!\!\! \int\limits_{\substack{\pa \Omega \\ \langle f'(\lambda), \vec{\nu}(\mx)\rangle  < 0}} V_+'(\lambda)\langle f'(\lambda), \vec{\nu}(\mx)\rangle \chi(\lambda,u_B(t,\mx)) \varphi(\mx)  d\mx d\lambda \nonumber\\&\qquad\qquad+t  \int_{\Omega}\int_0^b f'(\lambda) V_+'(\lambda) \chi(\lambda,v(t)) \varphi(\mx)  d\mx d\lambda+o(t).
\nonumber
\end{align}  Since for every $n$, the function $u_n(t,\cdot)$ has the same properties as the function $v$ from the above, we see that $u_n(t,\cdot)$ satisfies \eqref{hom5}. Therefore, as in the proof of Theorem \ref{mainthm} (more precisely relation \eqref{***}), we conclude that $(u_n)$ satisfies

\begin{align}
\label{semi+b-n}
&\int_{\R^d_+}\left(|u_n-k|_+ \pa_t \varphi+{\rm sgn}_+(u_n-k)(f(u_n)-f(k)) \, \nabla_{\mx} \varphi \right)d\mx dt\\&+ \int_{\R^d}|u_0-k|_+\varphi(0,\cdot) d\mx \nonumber \\& - \int_{0}^{b} \!\!\!\!\int\limits_{\substack{ \R^+\times \pa \Omega \\ \langle f'(\lambda),\vec{\nu}(\mx) \rangle < 0}} \!\!\!\!\varphi\,|\lambda-k|_+ \langle f'(\lambda),\vec{\nu}(\mx) \rangle \chi(\lambda, u_B(t,\mx)) d\gamma(\mx) dt d\lambda +{\cal O}(\frac{t}{n})\geq 0. \nonumber
\end{align} The left-hand side of the previous expression defines a non-positive distribution in $(t,\mx,k)\in \R^d_+\times \R$ and thus, it is a non-positive measure. We denote it by $m^n_+(t,\mx,k)$. Having this in mind, we get after differentiating \eqref{semi+b-n} with respect to $k$:

\begin{align}
\label{kin+b-n}
&\int_{\R^d_+}{\rm sgn}_+(u_n-k)_+ \left(\pa_t \varphi+f'(k) \, \nabla_{\mx} \varphi \right)d\mx dt+ \int_{\R^d}{\rm sgn}_+(u_0-k) \varphi(0,\cdot) d\mx 
\\& - \int_{0}^{b} \!\!\!\!\int\limits_{\substack{ \R^+\times \pa \Omega \\ \langle f'(\lambda),\vec{\nu}(\mx) \rangle < 0}} \!\!\!\!\varphi\,{\rm sgn}_+(\lambda-k) \langle f'(\lambda),\vec{\nu}(\mx) \rangle \chi(\lambda, u_B(t,\mx)) d\gamma(\mx) dt d\lambda \nonumber\\& +{\cal O}(\frac{t}{n})= \pa_k m^n_+(t,\mx,k), \nonumber
\end{align}where the above relation is understood in the sense of distributions in $k\in \R$. Letting $n\to \infty$ in \eqref{kin+b-n} along a subsequence for which both $({\rm sgn}_+(u_n-k))$ and $(m^n_+)$ weakly converge, we reach to \eqref{kin+}.

In order to get wanted relation for $V(\lambda)=V_-(\lambda)=|\lambda-k|_-$, remark that the function 
$$
-b\leq w=u-b \leq 0
$$ represents the weak solution to 
$$
\pa_t w + \Div_\mx f(w+b)=0,
$$ with the initial and boundary data
$$
-b \leq w_0=u_0-b \leq 0, \ \ -b \leq w_B=u_B-b\leq 0.
$$ If we apply the the transport-collapse procedure described in this section, then the corresponding sequence of approximate solutions has the form $(w_n)=(u_n-b)$ for $(u_n)$ defined in \eqref{aver} and \eqref{aver1}. We can thus repeat the arguments from \eqref{hom4} to conclude (keep in mind that now $\chi(\lambda,v(\mx))\leq 0$)

\begin{align}
\label{hom6}
&\int_{\Omega} \left(V_-(T(t)(v))(\mx)-V_-(v)(\mx) \right) \varphi(\mx) d\mx
\\& \leq \int_{-b}^0 \int\limits_{\substack{(\Omega-f'(\lambda+b)t)\backslash \Omega\\\langle f'(\lambda+b), \vec{\nu}_{\Omega-f'(\lambda+b)t}(\mx)\rangle< 0}} V_-'(\lambda) \chi(\lambda,u_B(0,\mx)-b) \varphi(\mx+t f'(\lambda+b))  d\mx d\lambda \nonumber\\
&-\int_{-b}^0 \int\limits_{\substack{\Omega\backslash (\Omega-f'(\lambda+b)t)\\\langle f'(\lambda+b), \vec{\nu}_{\Omega-f'(\lambda+b)t}(\mx)\rangle \geq 0}} V_-'(\lambda) \chi(\lambda,v(\mx)) \varphi(\mx)  d\mx d\lambda \nonumber\\
&+\int_{-b}^0 \int_{(\Omega-f'(\lambda+b)t)\cap \Omega} V_-'(\lambda) \chi(\lambda,v(\mx)) (\varphi(\mx+t f'(\lambda+b))-\varphi(\mx))  d\mx d\lambda.
\nonumber
\end{align} From here, as for $V_+$, we obtain \eqref{kin-}. \end{proof}

Denote by $p+$ and $p-$ the weak limits along appropriate subsequence of the sequences $({\rm sgn}_+(u_n-k))$ and $({\rm sgn}_{-}(u_n-k))$, respectively, defined in the proof of the previous theorem. It is not difficult to see that they satisfy conditions from \cite[Definition 3.1]{pan_tams}. This follows from the following theorem providing relation between Definition \ref{def2} and Definition \ref{def-panov} (from which the kinetic formulation given in \cite[Definition 3.1]{pan_tams} is derived).

\begin{theorem}
\label{eq}
A function $u$ satisfying Definition \ref{def2} satisfies Definition \ref{def-panov}.
\end{theorem}
\begin{proof}
It is enough to notice 
$$
|\langle f'(\lambda),\vec{\nu}(\mx)\rangle| \leq L,
$$ where $L$ is the constant such that $ \|f'\|_{\infty} \leq L$. Then, the last term in the left-hand side of \eqref{semi+b} satisfies for every $t\geq 0$ (the other terms there are the same as the corresponding ones from \eqref{semi+}):
\begin{align*}
&|\int_{0}^{b-a} \!\!\!\!\int\limits_{\substack{ \pa \Omega \\ \langle f'(\lambda-a),\vec{\nu}(\mx) \rangle < 0}} \!\!\!\!|\lambda-k-a|_+ \langle f'(\lambda-a),\vec{\nu}(\mx) \rangle \chi(\lambda, u_B(t,\mx)-a) d\gamma(\mx) d\lambda|
\\&\leq L \int_{\pa \Omega} \int_0^{b-a}|\lambda - k -a|_+ \, \chi(\lambda,u_B(t,\mx)-a)d\lambda d\gamma(\mx) 
\\&=L\int_{\pa \Omega}|u_B(t,\mx)-k|d\mx
\end{align*} from where we conclude that $u$ satisfies  \eqref{semi+}. The proof that $u$ satisfies \eqref{semi-} is the same.
\end{proof} Direct corollary of the previous theorem is existence and uniqueness of the function $u$ satisfying Definition \ref{def2}.

\begin{corollary}
There exists the function $u$ satisfying conditions of Definition \ref{def2} and it is  unique.
\end{corollary}
\begin{proof}
As we have already noticed, the functions $p+$ and $p-$ constructed in Theorem \ref{thm-bvp} (see the comments after the proof of the theorem) satisfy conditions of \cite[Definition 3.1]{pan_tams}. Therefore, according to \cite[Corollary 4.2]{pan_tams}, the function $p_+$ has the form $p_+(t,\mx,k)={\rm sgn}_+(u(t,\mx)-k)$ and therefore, $p_-(t,\mx,k)={\rm sgn}_{-}(u(t,\mx)-k)$ for some $u\in L^\infty(\R^+\times \Omega)$. Now, it is a standard fare to conclude that $u$ satisfies conditions of Definition \ref{def2} (it is the same as the proof of \cite[Theorem 3.3]{pan_tams}). According to Theorem \ref{eq} and the results from \cite{Otto}, we conclude that $u$ is a unique solution to \eqref{cl1-hom}, \eqref{ic1}, \eqref{bc1} in the sense of Definition \ref{def2}.
\end{proof}



Corresponding numerical examples are given below. It is
one-dimensional scalar conservation law defined on $[0,0.5]\times
[-1,1]$ with the flux
$f(x,u)=H_\eps(x)(1-u)(u+1)+4H_\eps(-x)(1-u)(u+1)$, where $H_\eps$
is a standard regularization of the Heaviside function with $\eps = 10^{-4}$. In the first simulation boundary conditions are
$u|_{x=-1}=0$, $u|_{x=1}=1$ and the initial condition is
$u|_{t=0}=H_\eps(x)$. In the second simulation boundary conditions are
$u|_{x=-1}=1$, $u|_{x=1}=0$ and the initial condition is  $u|_{t=0}=H_\eps(-x)$.

\begin{figure}[h]
\begin{center}$
\begin{array}{cc}
\includegraphics[width=2.5in]{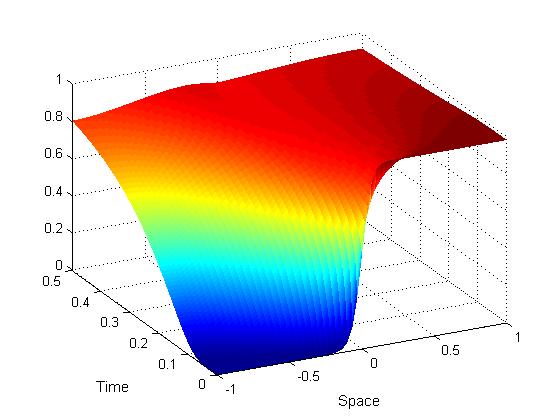} &
\includegraphics[width=2.5in]{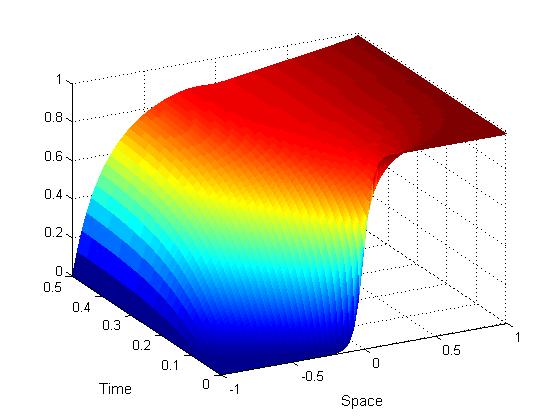}
\end{array}$
\end{center}
\caption{Cauchy problem (left) and boundary problem (right) with the initial condition $u_0(x) = H_\eps(x)$.}
\end{figure}
\begin{figure}[h]
\begin{center}$
\begin{array}{cc}
\includegraphics[width=2.5in]{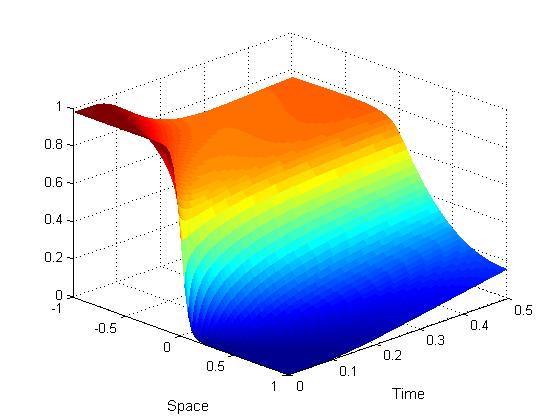} &
\includegraphics[width=2.5in]{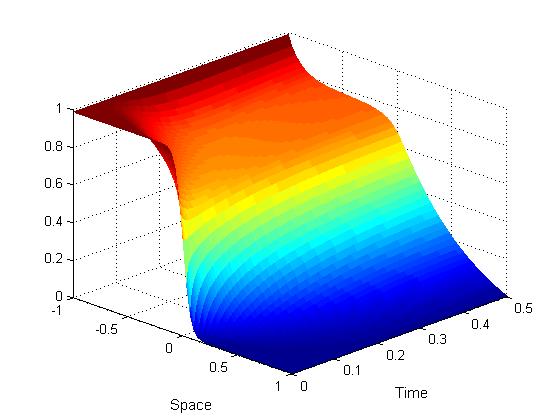}
\end{array}$
\end{center}
\caption{Cauchy problem (left) and boundary problem (right) with the initial condition $u_0(x) = H_\eps(-x)$.}
\end{figure}

\section{Acknowledgements}
The research is supported by the bilateral project \emph{Multiscale Methods and Calculus of Variations} between Croatia and Montenegro; by the Ministry of
Science of  Montenegro, project number 01-471; by the Croatian Science
Foundation, project number 9780 \emph{Weak convergence methods and applications (WeConMApp)}; and by the FP7 project \emph{Micro-local
defect functional and applications (MiLDeFA)} in the frame of the
program Marie Curie FP7-PEOPLE-2011-COFUND.

\end{document}